\documentclass[12pt]{amsart}
\usepackage[margin=1.2in]{geometry}
\usepackage{amssymb}
\usepackage{tikz}
\usetikzlibrary{calc}
\usetikzlibrary{decorations.markings}
\def\tS{\tilde S}
\def\Z{{\mathbb Z}}
\def\L{{\mathcal L}}
\def\hP{{\hat P}}
\def\C{{\mathcal C}}
\newtheorem{lemma}{Lemma}
\newtheorem{theorem}{Theorem}
\newtheorem{conjecture}{Conjecture}

\begin{document}
\title{The uncrossing partial order on matchings is Eulerian}
\author{Thomas Lam}\address{Department of Mathematics, University of Michigan,
2074 East Hall, 530 Church Street, Ann Arbor, MI 48109-1043, USA}
\email{tfylam@umich.edu}\thanks{T.L. was supported by NSF grant DMS-1160726.}
\begin{abstract}
We prove that the partial order $P_n$ on the set of matchings of $2n$ points on a circle, given by resolving crossings, is an Eulerian poset.  
\end{abstract}

\maketitle

\section{The Theorem}

Let $P_n$ denote the set of matchings on $2n$ points, labeled $1,2,\ldots,2n$ in order on a circle.  Each $\tau \in P_n$ can be represented by (usually many) medial graphs, or strand diagrams.  We say that a medial graph is {\it lensless} if any two strands intersect at most once.  The following lensless medial graph represents $\{(1,7),(2,9),(3,8),(4,10),(5,6)\} \in P_5$.

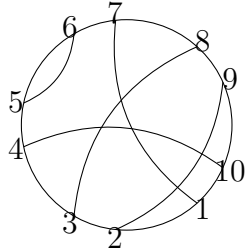
\begin{center}
\begin{tikzpicture}[scale=0.35]
\draw (0,0) circle (4cm);
\coordinate (t1) at (312:4);
\coordinate (t2) at (264:4);
\coordinate (t3) at (240:4);
\coordinate (t4) at (192:4);
\coordinate (t5) at (168:4);
\coordinate (t6) at (120:4);
\coordinate (t7) at (96:4);
\coordinate (t8) at (48:4);
\coordinate (t9) at (24:4);
\coordinate (t10) at (-24:4);
\draw (t1) to [bend left] (t7);
\draw (t3) to [bend left] (t8);
\draw (t2) to [bend right] (t9);
\draw (t5) to [bend right] (t6);
\draw (t4) to [bend left] (t10);

\node at (312:4.3) {$1$};
\node at (264:4.3) {$2$};
\node at (240:4.3) {$3$};
\node at (192:4.3) {$4$};
\node at (168:4.3) {$5$};
\node at (120:4.3) {$6$};
\node at (96:4.3) {$7$};
\node at (48:4.3) {$8$};
\node at (24:4.3) {$9$};
\node at (-24:4.3) {$10$};
\end{tikzpicture}
\end{center}

For $\tau \in P_n$ we let $c(\tau)$ denote the number of crossings in a lensless medial graph representing $\tau$.  The set $P_n$ can be equipped with a partial order obtained by resolving crossings: 
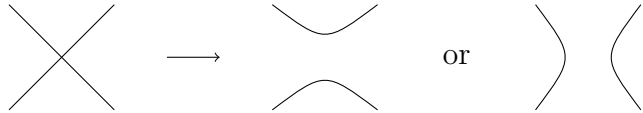
\begin{center}
\begin{tikzpicture}[scale = 0.7]
\node at (7.5,0) {or};
\draw (-1,-1) -- (1,1);
\draw (-1,1) -- (1,-1);

\draw[->] (2,0) -- (3,0);

\draw (4,-1) .. controls (5,-0.25) .. (6,-1);
\draw (4,1) .. controls (5,0.25) .. (6,1);
\draw (9,-1) .. controls (9.75,0) .. (9,1);
\draw (11,-1) .. controls (10.25,0) .. (11,1);
\end{tikzpicture}
\end{center}
We declare $\tau' \lessdot \tau$ if there is a lensless medial graph $G$ representing $\tau$ such that resolving a crossing in $G$ gives a lensless medial graph $G'$ representing $\tau'$.  The partial order $P_n$ is the transitive closure of these cover relations.

The poset $P_n$ is graded with rank function given by $c(\tau)$, and we refer the reader to \cite{ALT, Ken, Lam} for alternative descriptions of $P_n$.  The poset $P_n$ has a unique maximum element, and Catalan number $C_n$ of minimum elements.  Let $\hP_n$ denote $P_n$ with a minimum $\hat 0$ adjoined, where we declare that $c(\hat 0) = -1$.  Recall that a graded poset $P$ with a unique minimum and a unique maximum, is {\it Eulerian} if, for every interval $[x,y] \subset P$ where $x < y$, the number of elements of odd rank in $[x,y]$ is equal to the number of elements of even rank in $[x,y]$.  The following result was expected by many experts. 

\begin{theorem}\label{thm:main}
$\hP_n$ is an Eulerian poset.
\end{theorem}

Here is a picture of $\hP_3$:
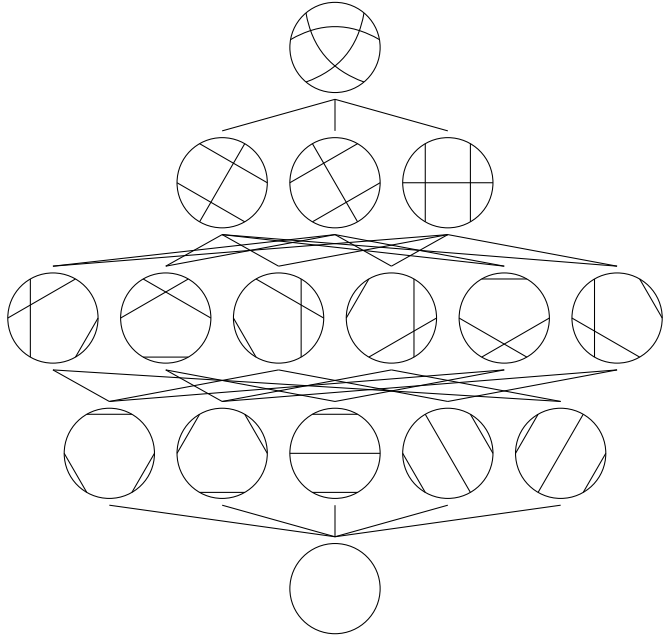
\begin{center}
\begin{tikzpicture}[scale = 0.3]
\draw (0,0) circle (2cm);
\coordinate (a1) at (130:2);
\coordinate (a2) at (50:2);
\coordinate (a3) at (10:2);
\coordinate (a4) at (-50:2);
\coordinate (a5) at (-130:2);
\coordinate (a6) at (170:2);
\draw (a1) to [bend right] (a4);
\draw (a2) to [bend left] (a5);
\draw (a3) to [bend right] (a6);

\draw (0,-2.3) -- (0,-3.7);
\draw (0,-2.3) -- (5,-3.7);
\draw (0,-2.3) -- (-5,-3.7);
\begin{scope}[shift = {(0,-6)}]
\draw (0,0) circle (2cm);
\coordinate (a1) at (120:2);
\coordinate (a2) at (60:2);
\coordinate (a3) at (0:2);
\coordinate (a4) at (-60:2);
\coordinate (a5) at (-120:2);
\coordinate (a6) at (180:2);
\draw (a1) -- (a4);
\draw (a2) -- (a6);
\draw (a3) -- (a5);
\draw (0,-2.3) -- (-17.5+1*5,-6+2.3);
\draw (0,-2.3) -- (-17.5+4*5,-6+2.3);
\draw (0,-2.3) -- (-17.5+5*5,-6+2.3);
\draw (0,-2.3) -- (-17.5+2*5,-6+2.3);
\end{scope}

\begin{scope}[shift = {(-5,-6)}]
\draw (0,0) circle (2cm);
\coordinate (a1) at (120:2);
\coordinate (a2) at (60:2);
\coordinate (a3) at (0:2);
\coordinate (a4) at (-60:2);
\coordinate (a5) at (-120:2);
\coordinate (a6) at (180:2);
\draw (a1) -- (a3);
\draw (a2) -- (a5);
\draw (a4) -- (a6);
\end{scope}
\begin{scope}[shift={(0,-6)}]
\draw (-5,-2.3) -- (-17.5+2*5,-6+2.3);
\draw (-5,-2.3) -- (-17.5+3*5,-6+2.3);
\draw (-5,-2.3) -- (-17.5+5*5,-6+2.3);
\draw (-5,-2.3) -- (-17.5+6*5,-6+2.3);
\end{scope}

\begin{scope}[shift = {(5,-6)}]
\draw (0,0) circle (2cm);
\coordinate (a1) at (120:2);
\coordinate (a2) at (60:2);
\coordinate (a3) at (0:2);
\coordinate (a4) at (-60:2);
\coordinate (a5) at (-120:2);
\coordinate (a6) at (180:2);
\draw (a2) -- (a4);
\draw (a3) -- (a6);
\draw (a1) -- (a5);
\end{scope}
\begin{scope}[shift={(0,-6)}]
\draw (5,-2.3) -- (-17.5+1*5,-6+2.3);
\draw (5,-2.3) -- (-17.5+3*5,-6+2.3);
\draw (5,-2.3) -- (-17.5+4*5,-6+2.3);
\draw (5,-2.3) -- (-17.5+6*5,-6+2.3);
\end{scope}

\begin{scope}[shift = {(-12.5,-12)}]
\draw (0,0) circle (2cm);
\coordinate (a1) at (120:2);
\coordinate (a2) at (60:2);
\coordinate (a3) at (0:2);
\coordinate (a4) at (-60:2);
\coordinate (a5) at (-120:2);
\coordinate (a6) at (180:2);
\draw (a3) -- (a4);
\draw (a5) -- (a1);
\draw (a6) -- (a2);
\end{scope}
\begin{scope}[shift={(0,-12)}]
\draw (-12.5,-2.3) -- (-15+1*5,-6+2.3);
\draw (-12.5,-2.3) -- (-15+5*5,-6+2.3);
\end{scope}

\begin{scope}[shift = {(-7.5,-12)}]
\draw (0,0) circle (2cm);
\coordinate (a1) at (120:2);
\coordinate (a2) at (60:2);
\coordinate (a3) at (0:2);
\coordinate (a4) at (-60:2);
\coordinate (a5) at (-120:2);
\coordinate (a6) at (180:2);
\draw (a4) -- (a5);
\draw (a6) -- (a2);
\draw (a3) -- (a1);
\end{scope}
\begin{scope}[shift={(0,-12)}]
\draw (-7.5,-2.3) -- (-15+2*5,-6+2.3);
\draw (-7.5,-2.3) -- (-15+3*5,-6+2.3);
\end{scope}

\begin{scope}[shift = {(-2.5,-12)}]
\draw (0,0) circle (2cm);
\coordinate (a1) at (120:2);
\coordinate (a2) at (60:2);
\coordinate (a3) at (0:2);
\coordinate (a4) at (-60:2);
\coordinate (a5) at (-120:2);
\coordinate (a6) at (180:2);
\draw (a5) -- (a6);
\draw (a1) -- (a3);
\draw (a2) -- (a4);
\end{scope}
\begin{scope}[shift={(0,-12)}]
\draw (-2.5,-2.3) -- (-15+1*5,-6+2.3);
\draw (-2.5,-2.3) -- (-15+4*5,-6+2.3);
\end{scope}

\begin{scope}[shift = {(2.5,-12)}]
\draw (0,0) circle (2cm);
\coordinate (a1) at (120:2);
\coordinate (a2) at (60:2);
\coordinate (a3) at (0:2);
\coordinate (a4) at (-60:2);
\coordinate (a5) at (-120:2);
\coordinate (a6) at (180:2);
\draw (a1) -- (a6);
\draw (a2) -- (a4);
\draw (a3) -- (a5);
\end{scope}
\begin{scope}[shift={(0,-12)}]
\draw (2.5,-2.3) -- (-15+2*5,-6+2.3);
\draw (2.5,-2.3) -- (-15+5*5,-6+2.3);
\end{scope}

\begin{scope}[shift = {(7.5,-12)}]
\draw (0,0) circle (2cm);
\coordinate (a1) at (120:2);
\coordinate (a2) at (60:2);
\coordinate (a3) at (0:2);
\coordinate (a4) at (-60:2);
\coordinate (a5) at (-120:2);
\coordinate (a6) at (180:2);
\draw (a1) -- (a2);
\draw (a3) -- (a5);
\draw (a4) -- (a6);
\end{scope}
\begin{scope}[shift={(0,-12)}]
\draw (7.5,-2.3) -- (-15+1*5,-6+2.3);
\draw (7.5,-2.3) -- (-15+3*5,-6+2.3);
\end{scope}

\begin{scope}[shift = {(12.5,-12)}]
\draw (0,0) circle (2cm);
\coordinate (a1) at (120:2);
\coordinate (a2) at (60:2);
\coordinate (a3) at (0:2);
\coordinate (a4) at (-60:2);
\coordinate (a5) at (-120:2);
\coordinate (a6) at (180:2);
\draw (a2) -- (a3);
\draw (a4) -- (a6);
\draw (a1) -- (a5);
\end{scope}
\begin{scope}[shift={(0,-12)}]
\draw (12.5,-2.3) -- (-15+2*5,-6+2.3);
\draw (12.5,-2.3) -- (-15+4*5,-6+2.3);
\end{scope}

\begin{scope}[shift = {(-10,-18)}]
\draw (0,0) circle (2cm);
\coordinate (a1) at (120:2);
\coordinate (a2) at (60:2);
\coordinate (a3) at (0:2);
\coordinate (a4) at (-60:2);
\coordinate (a5) at (-120:2);
\coordinate (a6) at (180:2);
\draw (a1) -- (a2);
\draw (a3) -- (a4);
\draw (a5) -- (a6);
\end{scope}
\begin{scope}[shift = {(-5,-18)}]
\draw (0,0) circle (2cm);
\coordinate (a1) at (120:2);
\coordinate (a2) at (60:2);
\coordinate (a3) at (0:2);
\coordinate (a4) at (-60:2);
\coordinate (a5) at (-120:2);
\coordinate (a6) at (180:2);
\draw (a2) -- (a3);
\draw (a4) -- (a5);
\draw (a1) -- (a6);
\end{scope}
\begin{scope}[shift = {(0,-18)}]
\draw (0,0) circle (2cm);
\coordinate (a1) at (120:2);
\coordinate (a2) at (60:2);
\coordinate (a3) at (0:2);
\coordinate (a4) at (-60:2);
\coordinate (a5) at (-120:2);
\coordinate (a6) at (180:2);
\draw (a1) -- (a2);
\draw (a3) -- (a6);
\draw (a4) -- (a5);
\end{scope}
\begin{scope}[shift = {(5,-18)}]
\draw (0,0) circle (2cm);
\coordinate (a1) at (120:2);
\coordinate (a2) at (60:2);
\coordinate (a3) at (0:2);
\coordinate (a4) at (-60:2);
\coordinate (a5) at (-120:2);
\coordinate (a6) at (180:2);
\draw (a2) -- (a3);
\draw (a4) -- (a1);
\draw (a6) -- (a5);
\end{scope}
\begin{scope}[shift = {(10,-18)}]
\draw (0,0) circle (2cm);
\coordinate (a1) at (120:2);
\coordinate (a2) at (60:2);
\coordinate (a3) at (0:2);
\coordinate (a4) at (-60:2);
\coordinate (a5) at (-120:2);
\coordinate (a6) at (180:2);
\draw (a3) -- (a4);
\draw (a2) -- (a5);
\draw (a1) -- (a6);
\end{scope}

\begin{scope}[shift = {(0,-24)}]
\draw (0,0) circle (2cm);
\begin{scope}[shift={(0,6)}]
\draw (0,-6 + 2.3) -- (-15+1*5,-2.3);
\draw (0,-6 + 2.3) -- (-15+2*5,-2.3);
\draw (0,-6 + 2.3) -- (-15+3*5,-2.3);
\draw (0,-6 + 2.3) -- (-15+4*5,-2.3);
\draw (0,-6 + 2.3) -- (-15+5*5,-2.3);
\end{scope}
\end{scope}

\end{tikzpicture}
\end{center}

The same result for the Bruhat order of a Weyl group was established by Verma \cite{Ver}, and our approach is similar.  Indeed, we rely on an embedding \cite{Lam} of $P_n$ into the (dual of the) Bruhat order of an affine symmetric group.  Theorem \ref{thm:main} further amplifies the analogy between Bruhat order and the uncrossing partial order $P_n$.

Like Bruhat order, the poset $P_n$ has has a topological interpretation.  In \cite{Lam} a compactification $E_n$ of the space of circular planar electrical networks with $n$ boundary points was constructed.  We have a stratification $E_n = \bigsqcup_{\tau \in P_n} E_\tau$ by {\it electroid cells}, and $\overline{E_\tau} = \bigsqcup_{\tau' \leq \tau} E_{\tau'}$.  The same partial order occurs in the study of the positive orthogonal Grassmannian and scattering amplitudes for ABJM \cite{HW,HWX}.  The partial order $P_n$ is also discussed in this context by Kim and Lee \cite{KL} who observed a special case of Theorem \ref{thm:main}.  It seems to be expected by experts that these spaces are homeomorphic to balls.

We remark that in \cite{ALT}, it is shown that $P_n$ is thin (intervals of length two are diamonds).  The following conjecture is probably widely expected (see the related \cite[Conjecture 3.3]{ALT}):
\begin{conjecture}
$\hP_n$ is lexicographically shellable.
 \end{conjecture}
 
An $EL$-shelling of $\hP_3$ is given as follows.  Denote the five non-crossing matchings by $\alpha_1,\alpha_2,\beta_1,\beta_2,\beta_3$, in order from left to right, as shown in the diagram above.  We label each edge by one of the symbols $\{\alpha_1,\alpha_2,\beta_1,\beta_2,\beta_3\}$. Label the covers $\hat 0 \lessdot \tau$ by $\tau$.  Label the cover $\eta \lessdot \tau$ where $c(\tau) = 1$, by $\eta'$ where $\eta' \neq \eta$ is uniquely defined by $\hat 0 \lessdot \eta' \lessdot \tau$.  Label the cover $\eta \lessdot \tau$ where $c(\tau) = 2$, by $\beta_i$, where $\beta_i$ is uniquely determined by the conditions $\beta_i < \tau$ and $\beta_i \not \leq \eta$.  Label the cover $\eta \lessdot \hat 1 = \tau_{\rm top}$ by $\beta_i$ where $\beta_i$ is uniquely determined by the condition $\beta_i \not \leq \eta$.

Finally, order the symbols by $\alpha_1 < \beta_1 < \beta_2 < \beta_3 < \alpha_2$.

\medskip
{\bf Acknowledgments.} It is a pleasure to thank Yu-tin Huang, Rick Kenyon, and Pavlo Pylyavskyy for conversations motivating this work.

\section{Medial graphs with lenses}
The notion of a medial graph representing a matching requires some care in the case that lenses are present. 

Suppose $G$ is any medial graph.  A {\it lens} in $G$ is a pair of strands that intersect twice or more.  By repeatedly applying the moves in Figures \ref{fig:Yang-Baxter} - \ref{fig:loop}, we can reduce $G$ to a lensless medial graph $G'$.  If $G'$ represents a matching $\tau \in P_n$, we shall say that $G$ represents $\tau \in P_n$ as well.  The following result follows from the interpretation of medial graphs in terms of electrical networks \cite{ALT,Ken,Lam}.

\begin{lemma}
Suppose $G$ is a medial graph with lenses and $G'$, $G''$ are two lensless graphs obtained from $G$ by the moves Figures \ref{fig:Yang-Baxter} - \ref{fig:loop}.  Then $G'$ and $G''$ represent the same matching $\tau \in P_n$.
\end{lemma}

\begin{lemma}\label{lem:moreuncrossing}
Suppose $G$ is a medial graph, possibly with lenses, that represents $\tau$.  Suppose $G'$ is obtained from $G$ by resolving any number of crossings in any way, and removing all closed interior loops that result.  If $G'$ represents $\tau'$, then $\tau \geq \tau'$.
\end{lemma}
\begin{proof}
In the language of \cite{Lam}, $G$ is the medial graph of some electrical network $\Gamma$, such that $\L(\Gamma) \in E_\tau$.  Then $G'$ is the medial graph of $\Gamma'$, where $\Gamma'$ is obtained from $\Gamma$ by either removing some edges, or contracting some edges.  (Interior loops correspond to isolated interior components of $\Gamma'$.)  Thus we must have $\L(\Gamma') \in \overline{E_{\tau}}$, and the result follows from \cite[Theorem 5.7]{Lam}, which states that $\overline{E_{\tau}} = \bigsqcup_{\eta \leq \tau} E_\eta$.
\end{proof}

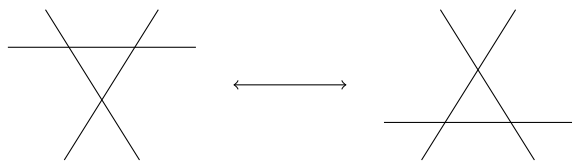
\begin{figure}
\begin{tikzpicture}[scale=0.5]
\draw (0,0) -- (5,0);
\draw (1,1) -- (3.5,-3);
\draw (4,1) -- (1.5,-3);
\draw[<->] (6,-1) -- (9,-1);
\begin{scope}[shift = {(10,-2)}]
\draw (0,0) -- (5,0);
\draw (1.5,3) -- (4,-1);
\draw (3.5,3) -- (1,-1);
\end{scope}
\end{tikzpicture}
\caption{The Yang-Baxter move}
\label{fig:Yang-Baxter}
\end{figure}

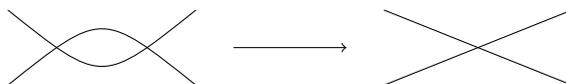
\begin{figure}

\begin{tikzpicture}[scale=0.5]
\draw (0,0) .. controls (2.5,2) .. (5,0);
\draw (0,2) .. controls (2.5,0) .. (5,2);
\draw [->] (6,1) -- (9,1);
\draw (10,0) -- (15,2);
\draw (10,2) -- (15,0);
\end{tikzpicture}

\caption{Lens removal}
\label{fig:lens}
\end{figure}

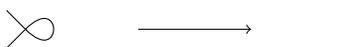
\begin{figure}

\begin{tikzpicture}[scale = 0.5]
\draw (1.5,-1.5) -- (2,-2);
\draw (2,-2) .. controls (3,-3) and (3,-1) .. (2,-2);
\draw (2,-2) -- (1.5,-2.5);
\draw [->] (5,-2) -- (8,-2);
\draw (10.5,-1.5) -- (10.5,-2.5);
\end{tikzpicture}

\caption{Loop removal}
\label{fig:loop}
\end{figure}

\section{The Proof}
For a subset $S \subset \hP_n$, write $\chi(S) = \sum_{\tau \in S} (-1)^{c(\tau)}$.  We need to show that $\chi([\tau,\eta]) = 0$ whenever $\tau < \eta$.

Let $\tS_{2n}$ denote the poset of affine permutations with period $2n$.  In \cite{Lam} we showed that there is an injection $\iota: P_n \hookrightarrow \tS_{2n}$ expressing $P_n$ as dual to an induced subposet of $\tS_{2n}$.  Write $g_\tau := \iota(\tau)$.  We refer the reader to \cite{Lam} for full details.

For $f \in \tS_{2n}$, we let 
$$
D_L(f) := \{i \in \Z/2n\Z \mid s_i f < f\}  \;\;\; \text{and} \;\;\; D_R(f) := \{i \in \Z/2n\Z \mid  fs_i < f\} 
$$
be the left and right descent sets of $f$.  The following result is standard, see \cite[Proposition 2.2.7]{BB}.

\begin{lemma}\label{L:Bruhat}
Suppose $f \leq g$ in $\tS_{2n}$.  If $i \in D_L(g) \setminus D_L(f)$ then $f \leq s_ig$ and $s_i f \leq g$.  If $i \in D_R(g) \setminus D_R(f)$ then $f \leq gs_i$ and $ fs_i \leq g$.\end{lemma}

For $\tau \in P_n$ and $i \in \Z/2n\Z$, we let
$$
i \in \begin{cases}  A(\tau) &\mbox{if the strands labeled $i$ and $i+1$ do not cross,}\\
 B(\tau) &\mbox{if the strands labeled $i$ and $i+1$ cross,} \\
 C(\tau) &\mbox{if $i$ is matched with $i+1$,}
 \end{cases}
$$
so that we have a disjoint union $\Z/2n\Z = A(\tau) \cup B(\tau) \cup C(\tau)$.   Define $s_i \cdot \tau$ by 
$$
g_{s_i \cdot \tau} = s_i g_\tau s_i.
$$
For example, if $i \in A(\tau)$ then
$s_i \cdot \tau$ is obtained from $\tau$ by adding a crossing between strands $i$ and $i+1$, close to the boundary at $i$ and $i+1$. Note that we have
$$
i \in \begin{cases} A(\tau) & \mbox{if $s_i g_\tau s_i < g_\tau$ or, equivalently, $s_i \cdot \tau \gtrdot \tau$,} \\
B(\tau) & \mbox{if $s_i g_\tau s_i > g_\tau$ or, equivalently, $s_i \cdot \tau \lessdot \tau$,} \\
C(\tau) & \mbox{if $s_i g_\tau s_i = g_\tau$ or, equivalently, $s_i \cdot \tau = \tau$.}
\end{cases}
$$

\begin{lemma}\label{lem:main}
Suppose $\tau \leq \eta$ in $P_n$.  If $i \in A(\tau) \cap B(\eta)$ then $\tau \leq s_i \cdot \eta$ and $s_i \cdot \tau \leq \eta$.
\end{lemma}
\begin{proof}
Since $i \in A(\tau)$ we have $i \in D_L(g_\tau) \cap D_R(g_\tau)$.  Since $i \in B(\eta)$ we have $i \notin D_L(g_\eta) \cup D_R(g_\eta)$.  By Lemma \ref{L:Bruhat}, we have $s_i g_\eta \leq g_\tau$.  
We also have, $i \notin D_R(s_ig_\eta)$ since $g_{s_i \cdot \eta} = s_ig_\eta s_i > s_i g_\eta$.  So by Lemma \ref{L:Bruhat} again, we have $s_i g_\eta s_i \leq g_\tau$, or equivalently, $g_{s_i\cdot \eta} \leq g_{ \tau}$, or equivalently, $\tau \leq s_i \cdot \eta$.  The proof of $s_i \cdot \tau \leq \eta$ is similar.
\end{proof}

The following result follows easily from the ``uncrossing" definition.
\begin{lemma}\label{lem:empty}
Suppose $\tau \leq \eta$ and $i \in A(\tau)$.  Then $i \notin C(\eta)$.
\end{lemma}

\begin{proof}[Proof of Theorem \ref{thm:main}]
We shall first prove the theorem for intervals $\tau < \eta$ of $P_n$; that is, intervals not involving $\hat 0$.  We proceed by descending induction on $c(\tau) + c(\eta)$.  The base case where $\eta$ is the maximal element of $P_n$ and $c(\tau) = \binom{n}{2}-1$ is clear.  Also if $c(\eta) -c(\tau) = 1$ the result is clear, so we may assume that $c(\eta) - c(\tau) \geq 2$.

Since $\tau$ is not the maximal element, $D_L(g_\tau)$ and $D_R(g_\tau)$ are non-empty, so $A(\tau)$ is non-empty (see also the proof of \cite[Lemma 4.14]{Lam}).  Let $i \in A(\tau)$.  

Case 1: $i \in A(\eta)$.  Then 
\begin{equation}\label{eq:twoterms}
[\tau,\eta] = [ \tau,s_i \cdot \eta] \setminus \{\sigma \mid \tau \leq \sigma \leq s_i \cdot \eta,  \sigma \not \leq \eta\}.
\end{equation}
We claim that 
\begin{equation}\label{eq:equals}
\{\sigma \mid \tau \leq \sigma \leq s_i \cdot \eta,  \sigma \not \leq \eta\} = \{\sigma \mid s_i \cdot \tau \leq \sigma \leq s_i \cdot \eta,  \sigma \not \leq \eta\}
\end{equation}
Suppose $\tau \leq \sigma \leq s_i \cdot \eta$ and $\sigma \not \leq \eta$.  If $i \in A(\sigma)$, then since $i \in B(s_i \cdot \eta)$, applying Lemma \ref{lem:main} to $\sigma < s_i \cdot \eta$ we would obtain $\sigma \leq \eta$, contradicting our assumption.  If $i \in C(\sigma)$, then Lemma \ref{lem:empty} would be violated.  Thus $i \in B(\sigma)$, or $s_i \cdot \sigma \lessdot \sigma$.  Now apply Lemma \ref{lem:main} to $\tau \leq \sigma$ to obtain $s_i \cdot \tau \leq \sigma$.  This proves \eqref{eq:equals}.

But we have
$$
\{\sigma \mid s_i \cdot \tau \leq \sigma \leq s_i \cdot \eta,  \sigma \not \leq \eta\} = [s_i \cdot \tau,s_i \cdot \eta] \setminus  [s_i \cdot \tau,\eta]
$$
so by induction $\chi(\{\sigma \mid s_i \cdot \tau \leq \sigma \leq s_i \cdot \eta,  \sigma \not \leq \eta\}) = \chi([s_i \cdot \tau,s_i \cdot \eta] ) - \chi([s_i \cdot \tau,\eta]) = 0$.  By induction again, $\chi([\tau,s_i \cdot \eta])= 0$ so using \eqref{eq:twoterms} we obtain $\chi([\tau,\eta]) = 0$.

Case 2: $i \in B(\eta)$.  We apply Lemma \ref{lem:main} to see that $\sigma \in [\tau,\eta]$ implies $s_i \cdot \sigma \in [\tau,\eta]$.  By Lemma \ref{lem:empty}, it follows that $\sigma \mapsto s_i\cdot \sigma$ is an involution which swaps elements of odd rank with elements of even rank.  Thus $\chi([\tau,\eta]) = 0$.

Case 3: $i \in C(\eta)$.  This is impossible by Lemma \ref{lem:empty}.

Thus we have shown that $\chi([\tau,\eta]) = 0$ for intervals $\tau < \eta$ not involving $\hat 0$.

\bigskip

Now suppose $\tau = \hat 0$ and $\hat 0 < \eta$.  We may suppose that $c(\eta) \geq 1$.  Then $B(\eta) \neq \emptyset$.  Let $i \in B(\eta)$.  By Lemma \ref{lem:main}, the map $\sigma \mapsto s_i \cdot \sigma$ establishes an involution on the set $\{\sigma \in (\hat 0, \eta] \mid i \notin C(\eta)\}$, and this involution swaps the parity of $c(\sigma)$.  Let 
$$
S = \{\sigma \in [\hat 0, \eta) \mid i \in C(\eta) \text{ or } \sigma = \hat 0\}.
$$ It thus suffices to show that $\chi(S) = 0$.  We claim that $S$ has a unique maximal element $\kappa$.  This would complete the proof since then $S = [\hat 0, \kappa]$ and we may proceed by induction.

We now construct $\kappa \in P_n$.  Let $G$ be a medial graph representing $\eta$.  The strands $p_i$ and $p_{i+1}$ beginning at $i$ and $i+1$ cross each other since $i \in B(\eta)$.  We shall assume $p_i$ and $p_{i+1}$ cross each other at $q$ before intersecting any other strands.  Let $G'$ be obtained from $G$ by uncrossing $q$ so that $i$ is matched directly with $i+1$.  Let $\kappa \in P_n$ be represented by $G'$.  

Suppose $\sigma \in S$.  We need to show that $\sigma \leq \kappa$.  A (lensless) medial graph $G''$ representing $\sigma$ can be obtained from $G$ by uncrossing some subset $\C$ of the crossings of $G$ (see the comment after \cite[Lemma 4.11]{Lam}).  Since $i \in C(\sigma)$ we must have $q \in \C$.  If $q$ is resolved in $G''$ in the same way as in $G'$ then we are done: $G''$ can be obtained from $G'$ by a number of uncrossings, and Lemma \ref{lem:moreuncrossing} gives $\sigma \leq \kappa$.  Now suppose $q$ is resolved in $G''$ in the direction different to the one in $G'$.  We then observe that the medial graph $G'''$ obtained from $G''$ by resolving the uncrossing $q$ in the other direction also represents the same matching $\sigma$ (a closed interior loop will appear in $G'''$, which can be removed).  This completes the proof that $S = [0,\kappa]$, and thus the theorem.
\end{proof}

\end{document}